\documentclass[12pt]{amsart}
\usepackage{fullpage, bbm}
\usepackage{amsfonts}
\usepackage{amssymb}
\usepackage{amsmath}
\usepackage{amsthm}
\usepackage{latexsym}

\newtheorem{theorem}{Theorem}[section]

\newtheorem{lemma}[theorem]{Lemma}
\newtheorem{conjecture}[theorem]{Conjecture}

\newcommand{\one}{\mathbbm{1} }
\newcommand{\RR}{\mathbb{R}}
\newcommand{\sinc}{\mathrm{sinc}}

\begin{document}

\author{Michael R. Tehranchi \\
University of Cambridge }
\address{Statistical Laboratory\\
Centre for Mathematical Sciences\\
Wilberforce Road\\
Cambridge CB3 0WB\\
UK}
\email{m.tehranchi@statslab.cam.ac.uk}

\title{Inequalities for the Gaussian measure of convex sets}
\date{\today}
\thanks{\textit{Key words or phrases.}   Gaussian measure, correlation inequality, log-concavity}
\thanks{\textit{2010 AMS subject classification} 60E15,  28C20.}

\begin{abstract}
This note presents families of inequalities for the Gaussian measure of convex sets
which extend the recently proven Gaussian correlation inequality in various directions. 
\end{abstract}

\maketitle

\section{Introduction and statement of results}
Let $\gamma$ be the standard Gaussian on $\RR^n$,  
defined by
$$
\gamma(K) = \int_K  \tfrac{ 1}{(2\pi)^{n/2}} e^{- \tfrac{1}{2} \|x \|^2 } dx
$$
for   Lebesgue  measurable $K \subseteq \RR^n$.  

Recently Royen \cite{royen} proved  that 
\begin{equation}\label{eq:gci}
 \gamma(A) \gamma(B)   \le \gamma( A \cap B)
\end{equation}
for all dimensions $n$ and all symmetric convex sets $A, B \subseteq \RR^n$.
The Gaussian correlation inequality \eqref{eq:gci} was previously known as the Gaussian correlation conjecture and was
an open problem for over 50 years.  See
  the paper of Lata\l{}a \& Matlak \cite{LM}  for a discussion of Royen's proof.  

The purpose of this note is to offer evidence in support of the following strengthening of 
inequality \eqref{eq:gci}.  We will use the notation  
$$
A+B = \{a+b: a \in A, b \in B\}
$$
for the Minkowski sum of two sets.

\begin{conjecture}\label{conj}
The inequality 
\begin{equation}\label{eq:conj}
 \gamma(A) \gamma(B)   \le \gamma( A \cap B) \gamma(A+B)
\end{equation}
holds for all dimensions $n$ and all symmetric convex sets $A, B \subseteq \RR^n$.
\end{conjecture}

It is obvious that inequality \eqref{eq:conj} holds in dimension $n=1$. More
generally, the inequality holds whenever $A \subseteq B$, since in this case $A = A \cap B$
and $B \subseteq A+B$.
Also note that inequality \eqref{eq:conj} holds with equality for 
any dimension $n > 1$ whenever there is a dimension $m < n $ and   sets $\tilde A \subseteq \RR^m$
and $\tilde B \subseteq \RR^{n-m}$ such that $A = \tilde A \times \RR^{n-m}$ and $B =
\RR^{m} \times \tilde B$, since in this case $A \cap B = \tilde A \times \tilde B$
and $A+B = \RR^n$.

 Dar \cite{dar} proved that the similar-looking inequality
\begin{equation}\label{eq:dar}
\mathrm{Leb}(A) \mathrm{Leb}(B) \le \mathrm{Leb}( A \cap B) \mathrm{Leb}(A + B)
\end{equation}
holds for all symmetric convex sets $A, B \subseteq \RR^n$, where
 $\mathrm{Leb}$ is the Lebesgue measure on $\RR^n$.  Since we have the inequality
$$
 \tfrac{ 1}{(2\pi)^{n/2}} e^{-\tfrac{1}{2} r_K^2} \ \mathrm{Leb}(K) \le \gamma(K) \le \tfrac{ 1}{(2\pi)^{n/2}}   \mathrm{Leb}(K) 
$$
for all bounded measurable $K \subset \RR^n$, where $r_K = \sup\{ \| x \|: x \in K \}$ is the
radius of the smallest ball containing $K$, inequality \eqref{eq:dar} implies
\begin{equation}\label{eq:dar2}
\gamma(A) \gamma(B) \le \gamma( A \cap B) \gamma(A + B) e^{ \tfrac{1}{2}(r_A+r_B)^2  + \tfrac{1}{2}(r_A \wedge r_B)^2   }.
\end{equation}
Inequality \eqref{eq:dar2} does not prove Conjecture \ref{conj}, but it does indicate
that the conjecture is plausible.
Furthermore, even if Conjecture \ref{conj} turns out not to be true, inequality \eqref{eq:dar2}
shows that the correlation inequality \eqref{eq:gci} can be improved when $A$ and $B$ are
contained in a sufficiently small ball.  
Indeed, the right-hand side of inequality \eqref{eq:dar2} is smaller than the right-hand side
of the correlation inequality \eqref{eq:gci} when $r_A$ and $r_B$ are sufficiently small.

Schechtman, Schlumprecht \& Zinn \cite[Proposition~3]{SSZ} proved
the related inequality that 
\begin{equation}\label{eq:SSZ-ineq}
\gamma(   A) \gamma(  B) \le \gamma\big( \sqrt{2} (A \cap B) \big) \  \gamma\big( \tfrac{1}{\sqrt{2}}(A+B) \big)
\end{equation} 
for symmetric convex $A,B \subseteq \RR^n$. 
Using the fact that the map
$$
t \mapsto t^{-n} \gamma(t K) = \int_{K}    \tfrac{ 1}{(2\pi)^{n/2}} e^{- \tfrac{t^2}{2} \|x \|^2 } dx
$$
is decreasing for any measurable $K \subseteq \RR^n$, inequality
\eqref{eq:SSZ-ineq}   implies
\begin{equation}\label{eq:SSZ-2}
  \gamma(A) \gamma(B) \le 2^{n/2} \gamma(A \cap B)    \gamma\big( \tfrac{1}{\sqrt{2}}(A+B) \big)
\end{equation}
as was observed by Schechtman, Schlumprecht \& Zinn.   Note that since $\tfrac{1}{\sqrt{2}} < 1$, 
the right-hand side
of inequality \eqref{eq:SSZ-2} is larger than the right-hand side of the  conjectural 
inequality \eqref{eq:conj}.   Also note that replacing $A$ and $B$ with $tA$ and $tB$
and sending $t \downarrow 0$ in either inequality \eqref{eq:SSZ-ineq} or \eqref{eq:SSZ-2}
recovers inequality \eqref{eq:dar}.
 
The  new result of this paper is the following:
\begin{theorem}\label{th:main}
The inequality 
\begin{equation}\label{eq:mine}
 \gamma(A) \gamma(B)   \le  \left( 1-s   \right)^{-n/2} 
\gamma \left( \sqrt{\tfrac{2(1-s )}{1+t}} (A \cap B) \right)  \gamma \left( \sqrt{ \tfrac{1-s }{2(1-t)}}  (A + B) \right) 
\end{equation}
holds for all dimensions $n$ and all symmetric convex sets $A, B \subseteq \RR^n$ and 
all $ \sqrt{s} \le t < 1$.
\end{theorem}
The proof of Theorem \ref{th:main} uses a stronger form of the Gaussian correlation inequality \eqref{eq:gci} 
which already appears in Royen's paper,  
as well as ideas appearing in the papers of  Shao~\cite{shao} and 
Schechtman, Schlumprecht \& Zinn.   We present the proof in the next section.

Note that setting $s=0$ in inequality \eqref{eq:mine} yields the dimension-independent
family of inequalities
$$
 \gamma(A) \gamma(B)   \le   
\gamma \left( \sqrt{\tfrac{2}{1+t}} (A \cap B) \right)  \gamma \left( \tfrac{1}{\sqrt{2(1-t)}} (A + B) \right) 
$$
which holds for all $0 \le t < 1$.  This family  interpolates
between Schechtman, Schlumprecht \& Zinn's inequality
\eqref{eq:SSZ-ineq} corresponding to $t=0$ and Royen's inequality \eqref{eq:gci} corresponding to the 
limit $t \uparrow 1$.  Setting $t=1/2$ yields
\begin{equation}\label{eq:mine2.1}
 \gamma(A) \gamma(B)   \le   
\gamma \left( \tfrac{2}{\sqrt{3}} (A \cap B) \right)  \gamma \left( A + B  \right) 
\end{equation} 
Note that since $\frac{2}{\sqrt{3}} > 1$, the right-hand side of inequality \eqref{eq:mine2.1}
is larger than the right-hand side of the conjectural inequality  \eqref{eq:conj}.

Note that by setting $s=  \tfrac{1}{2}(1-t) $ for $1/2 \le t < 1$  in inequality \eqref{eq:mine}, we have the family of inequalities
$$
 \gamma(A) \gamma(B)   \le    \left(\tfrac{2}{1+t} \right)^{n/2} 
\gamma \left( A \cap B  \right)  \gamma \left( \sqrt{\tfrac{1+t}{4(1-t)}} (A + B) \right).
$$
 Again, the limit $t \uparrow 1$ recovers inequality \eqref{eq:gci}.  
Setting $t = 1/2$ yields 
\begin{equation}\label{eq:mine4}
 \gamma(A) \gamma(B)   \le    \left(\tfrac{4}{3} \right)^{n/2} 
\gamma \left( A \cap B  \right)  \gamma \left(  \tfrac{\sqrt{3}}{2} (A + B) \right).
\end{equation}
Note that since $\tfrac{1}{\sqrt{2}} < \tfrac{\sqrt{3}}{2} < 1 $ the right-hand side of inequality \eqref{eq:mine4}   
is larger than the right-hand side of the conjectural inequality  \eqref{eq:conj}, but 
it is smaller than the right-hand side of inequality \eqref{eq:SSZ-2},
and therefore improving on the result of Schechtman, Schlumprecht \& Zinn.  Finally, setting $t = 3/5$  yields
$$
 \gamma(A) \gamma(B)   \le    \left(\tfrac{5}{4} \right)^{n/2} 
\gamma \left( A \cap B  \right)  \gamma \left(  A + B  \right)
$$
which improves upon inequality \eqref{eq:dar2} when either $A$ or $B$ is  unbounded.

Finally, note that by setting $s =  2t-1 $ for 
$1/2 \le t < 1$ in inequality \eqref{eq:mine}, we have the family of inequalities
$$
 \gamma(A) \gamma(B)   \le    [2 (1-t)]^{-n/2} \gamma \left(\sqrt{\tfrac{4(1-t)}{1+t} } (A \cap B)  \right)
   \gamma \left(    A + B  \right).
$$%
Sending $t \uparrow 1$ yields
\begin{align*}
 \gamma(A) \gamma(B)   &\le   \tfrac{1}{(2\pi)^{n/2}} \mathrm{Leb} \left( A \cap B  \right)  \gamma \left(    A + B  \right) \\
&\le\gamma \left( A \cap B  \right)  \gamma \left(  A + B  \right) e^{ \tfrac{1}{2}(r_A \wedge r_B)^2  }
\end{align*}
 which again improves upon inequality \eqref{eq:dar2} when either $A$ or $B$ is  bounded.

\section{The proof}\label{se:proofs}
Fix the dimension $n$, and for $0 \le t \le 1$
 let $\gamma_t$ denote the measure on $\RR^n \times \RR^n$ 
which interpolates between  $\gamma_0(K \times L) = \gamma(K) \gamma(L)$ and
$\gamma_1(K \times L) = \gamma(K \cap L)$ given explicitly by the formula
$$
\gamma_t (H)  = \int_H \tfrac{1}{(1-t^2)^{n/2}(2\pi)^{n}} 
e^{-\tfrac{1}{2(1-t^2)}(\|x\|^2 - 2 t \langle x,  y \rangle + \|y\|^2 ) } dx \  dy 
$$
for measurable $H \subseteq \RR^{n} \times \RR^{n}$,
 where $\langle x, y \rangle   = x_1y_1 + \ldots + x_n y_n$ denotes the standard
inner product on $\RR^n$.   We will need a few observations about the measure
$\gamma_t$.

\vskip .25cm

\noindent \textit{ Observation 1.} Fix a measurable set $H \in \RR^n \times \RR^n$ 
with the symmetry property that  $(x,y) \in H$ implies $(x,-y) \in H$.  Let
\begin{align*}
f(t) &=  (1-t^2)^{-n/2} \gamma_t\big( \sqrt{1-t^2} H \big) \\
 & =  \int_{H} \tfrac{1}{(2\pi)^n} e^{ - \tfrac{1}{2} (\|x\|^2  - 2 t \langle x,  y \rangle + \|y\|^2) } dx \ dy \\
 & =  \int_{H} \tfrac{1}{(2\pi)^n} e^{ - \tfrac{1}{2} (\|x\|^2 + \|y\|^2) } \cosh( t \langle x,  y \rangle) dx \ dy 
\end{align*}
 where we have used the symmetry property of $H$ to go from the second to third line.  Hence, we 
have the identity 
$$
f'(t) = \int_{H} \tfrac{1}{(2\pi)^n} e^{ - \tfrac{1}{2} (\|x\|^2 + \|y\|^2) } 
\langle x,  y \rangle \sinh( t \langle x,  y \rangle)  \ dx \ dy.
$$
Since $\theta  \sinh \theta \ge 0$ for all real $\theta$, the function $f$ is increasing.   A variation of this argument also appears in the paper
of  Shao~\cite[Theorem~1.1]{shao}.

\vskip .25cm

\noindent \textit{Observation 2.} Inspection of Royen's proof \cite[equation (2.3)]{royen} of the Gaussian correlation inequality \eqref{eq:gci}
shows that the map
$$
t \mapsto \gamma_t(A \times B)  
$$ 
is increasing on $[0,1]$ for all symmetric convex $A,B \subseteq \RR^n$.  This monotonicity
 property was already known for the special case of dimension $n=2$ by the result of Pitt \cite[Theorem 3]{pitt}.
In Appendix \ref{ap:sinc} we provide an interesting reformulation of this monotonicity
property in terms of the function $\sinc \ x = \frac{\sin x}{x}$.

\vskip .25cm

\noindent \textit{Observation 3.1.}  Note the standard fact about Gaussian measure that
$$
\gamma_t( K \times L ) =  \gamma_0 \left\{ (x,y):  \sqrt{\tfrac{1+t}{2}}
 x +  \sqrt{\tfrac{1-t}{2}} y \in K, \sqrt{\tfrac{1+t}{2}}x - \sqrt{\tfrac{1-t}{2}} y  \in L \right\}.
$$

\vskip .25cm

\noindent \textit{Observation 3.2.}  Fix symmetric convex $A, B \subseteq \RR^n$ and real constant $p$
and 
note that 
$$\gamma_0 \{ (x,y):  x +  p y \in A, x - py  \in B \}=   \int_{\RR^n} h(y)  d\gamma(y)
$$
 where
 $$
 h(y) = \gamma\left\{ (  A - py) \cap (  B + p y)  \right\}.
 $$
 The function $h$ is log-concave by the log-concavity of the Gaussian density, the assumed convexity of $A$ and $B$ and
Pr\'ekopa's theorem  \cite[Theorem 6]{prekopa}.  For completeness a statement of this important result is included in the appendix.  
Since $h$ is even by the assumed
 symmetry of $A$ and $B$ we have
 $$
 h(y) \le h(0)  = \gamma(A \cap B).
 $$
 for all $y \in \RR^n$.  Furthermore, $h(y) > 0$ only when 
 $$
(  A - py) \cap (  B + p y)\ne \emptyset,
 $$
 that is, when there exist points $a \in A$ and $b \in B$
 such that
 $$
a - p y =  b  +  p   y  
 $$
 and hence 
 $$
 y  = \tfrac{1}{2p} (a-b) \in \tfrac{1}{2p} (A+B),
 $$
 again by the symmetry of $B$.   
Therefore
\begin{align*}
 \int_{\RR^n}  h(y) d\gamma (y) &\le \int_{y: h(y)>0}  h(0) d\gamma  (y) \\
 & \le \gamma  ( A \cap B) \ \gamma \big( \tfrac{1}{2p} (A+B) \big).   
 \end{align*}

\vskip .25cm

\noindent \textit{Observation 3.3.}  Fix $0 \le t < 1$ and symmetric convex $A,B$.
Combining Observations 3.1 and 3.2 yields
\begin{align*}
\gamma_t(A \times B) &= \gamma_0 
\left\{ (x,y):  
 x +  \sqrt{\tfrac{1-t}{1+t}} y \in \sqrt{\tfrac{2}{1+t} } A, x - \sqrt{\tfrac{1-t}{1+t}} y  \in \sqrt{\tfrac{2}{1+t}} B \right\} \\
& \le \gamma \left( \sqrt{\tfrac{2}{1+t}} (A \cap B) \right)  \gamma \left( \tfrac{1}{\sqrt{2(1-t)} }(A + B) \right) 
\end{align*}
The idea to use the elementary Observation 3.1 and the more sophisticated Observation 3.2 to bound the Gaussian
measure of an intersection was taken from Schechtman, Schlumprecht \& Zinn \cite[Proposition~3]{SSZ}.
In fact,   Dar \cite[Observation (4)]{dar}  also employed the analogue of Observation 3.2 to bound 
the Lebesgue measure of an intersection, and indeed this type of argument seems to have  
originated in the paper of Rogers \& Shephard \cite{RS}.

 \vskip .5cm

To prove Theorem \ref{th:main},  fix $  \sqrt{s} \le t <1$ and 
convex symmetric sets $A,B$.  We have the following series
of inequalities:
  \begin{align*}
   \gamma (A) \gamma (B) & = \gamma_0 (A \times B) \\
	& \le (1-s )^{-n/2}  \gamma_{\sqrt{s}} \big( \sqrt{1-s } (A \times B) \big) \\
		& \le (1-s )^{-n/2}  \gamma_t \big( \sqrt{1-s } (A \times B) \big) \\
	& \le (1-s )^{-n/2} \gamma \left( \sqrt{\tfrac{2(1-s )}{1+t}} (A \cap B) \right)  
	\gamma \left( \sqrt{\tfrac{1-s }{2(1-t)} }(A + B) \right) 
\end{align*}
as desired. \qed

\section{Appendix: A $\sinc$ reformulation}\label{ap:sinc}
In this appendix, we provide an interesting equivalent reformulation of Royen's 
result that the map $t \mapsto \gamma_t(A \times B)$ is increasing
for symmetric convex set $A$ and $B$, where the interpolation measure $\gamma_t$ is 
defined in section \ref{se:proofs}.

We will use the notation $\sinc : \RR^n \to \RR$ defined by
$$
\sinc(x_1, \ldots, x_n) = \prod_{i=1}^n \frac{ \sin x_i}{x_i}.
$$

\begin{theorem}\label{th:sinc}
For all $0 \le t < 1$ and $n \times n$ matrices $P$ and $Q$ we have
$$
\int \sinc( Px) \sinc(Q y) \langle x, y \rangle \sinh( t \langle x, y \rangle )   d \gamma (x) d \gamma (y) \ge 0.
$$
\end{theorem}

To prove Theorem \ref{th:sinc}, we will need a lemma about Gaussian Fourier transforms. 
 We note that the idea to study Gaussian (and more general) correlation
inequalities via Fourier analysis has appeared in the paper of Koldobsky \& Montgomery-Smith~\cite{KM}.  We need some notation.
For integrable $f:\RR^n \to \mathbb{C}$ define its Fourier transforms $\hat{f}:\RR^n \to \mathbb{C} $ by
$$
\hat{f}(u) = \tfrac{1}{(2\pi)^{n/2}} \int e^{\mathrm{i} \langle u, x \rangle } f(x) dx
$$
as usual.

\begin{lemma}\label{th:plancherel}  If $f$ and $g$ are integrable then
$$
\int f(x) g(y) d\gamma_{t}(x,y) = \int \hat{f}(u) \hat{g}(v) e^{-t \langle u, v \rangle } d \gamma(u) d \gamma(v)
$$
for all $0 \le t < 1$.
\end{lemma}

\begin{proof}
This is essentially an application of Plancheral's identity.
The proof amounts to writing $\hat f$ and $\hat g$ in terms of their respective
Fourier integrals, and since $f$ and $g$ are assumed integrable, Fubini's
theorem can be applied. The result is a consequence of the well-known formula
$$
\int e^{\mathrm{i}(  \langle u, x \rangle + \langle v, y \rangle)- t \langle u, v \rangle} d \gamma(u) d \gamma(v)
 = 
\tfrac{1}{(1-t^2)^{n/2}} 
e^{-\frac{1}{2(1-t^2)}(\|x\|^2 - 2 t \langle x, y \rangle + \|y\|^2 )}.
$$
\end{proof}

\begin{lemma}\label{th:fourier}
Let $C = [-1,1]^n = \{ x \in \RR^n, \max_i |x_i| \le 1 \}$ and set
 $A = P^\top C$ and $B = Q^\top C$ for $n\times n$ matrices $P$ and $Q$.
Then
$$
\gamma_t (A \cap B)  =  |\det(P) \det(Q)|   \int \sinc( Px) \sinc(Q y) \cosh(
  t\ \langle x, y \rangle)   d \gamma (x) d \gamma (y).
$$
\end{lemma}

\begin{proof}
Note that $\widehat{ \one_C }(s) = \sinc(s)$ and hence $\widehat{ \one_{P^\top C}}(s) = |\det P| \sinc(P s)$.
By Lemma \ref{th:plancherel}, we have
\begin{align*}
\gamma_t(A \times B) &=  |\det(P) \det(Q)| \int \sinc( Px) \sinc(Q y) e^{-t \langle x, y \rangle} d \gamma(x) d \gamma(y).
\end{align*}
The result follows since $C$ is symmetric and $\sinc$ is even. 
\end{proof}

The proof of Theorem \ref{th:sinc} follows from differentiating the expression in Lemma \ref{th:fourier}
and applying Royen's result on the monotonicity of $t \mapsto \gamma_t(A \times B)$.  
Notice that since   convex sets can be 
approximated by polyhedra, Theorem \ref{th:sinc} is in fact 
  equivalent to Royen's monotonicity result.

 \section{Appendix: Log-concave functions}\label{ap:prekopa} 
 In this appendix we recall some familiar notions involving log-concavity.  A non-negative
 function $g$ on $\RR^n$ is called log-concave if
 $$
 g( \theta x + (1-\theta) y) \ge g(x)^{\theta} g(y)^{1-\theta}
 $$
 for any $0 \le \theta \le 1$ and $x,y \in \RR^n$.   In particular, the indicator function of 
 a convex set is log-concave.  The following fundamental result
 is due to Pr\'ekopa  \cite[Theorem 6]{prekopa}.
 
 \begin{theorem}
 Suppose that the function $g$  on $\RR^{m+n}$ is log-concave.  Then the function $h$ on $\RR^n$
 defined by
 $$
 h(y) = \int_{\RR^m} g(x,y) dx
 $$
 is also log-concave.
 \end{theorem}
In section \ref{se:proofs} we appeal to Pr\'ekopa's theorem with the log-concave function
$$
g(x,y) =  \tfrac{1}{(2\pi)^{n/2}} \one_{\{ (x,y):  x + py \in A, x-py \in B \}} e^{-\| x\|^2/2 }.  
$$

\end{document}